\documentclass[10pt,a4paper,reqno]{amsart}
\usepackage{amsthm}
\usepackage{amsmath}
\usepackage{amssymb}
\usepackage{graphicx,color}

\usepackage[font=small]{caption}
\usepackage{subcaption}

\usepackage{multirow}
\usepackage[shortlabels]{enumitem}

\usepackage{tikz}
\usepackage{pgf,tikz}
\usepackage{mathrsfs}
\usetikzlibrary{arrows}
\usetikzlibrary{snakes}
\usetikzlibrary{arrows}
\usetikzlibrary{calc}
\usetikzlibrary{patterns}
\usepackage{tkz-fct}

\makeatletter
\theoremstyle{plain}
\newtheorem{thm}{Theorem}
  \theoremstyle{definition}
  
  \theoremstyle{remark}
  \newtheorem{rem}[thm]{Remark}
  \theoremstyle{plain}
  \newtheorem{prop}[thm]{Proposition}
  \theoremstyle{plain}
  \newtheorem{lem}[thm]{Lemma}
  \theoremstyle{plain}
  \newtheorem{cor}[thm]{Corollary}
 \theoremstyle{definition}
  
  \theoremstyle{remark}
  \newtheorem*{rem*}{Remark}

  \theoremstyle{definition}

\usepackage{amsfonts}
\usepackage{mathrsfs}

\newtheorem*{question*}{\it{QUESTION}}

\theoremstyle{plain}
\newcommand{\N}{\mathbb{N}}
\newcommand{\R}{{\mathbb{R}}}
\newcommand{\C}{{\mathbb{C}}}

\newcommand{\Z}{{\mathbb{Z}}}
\newcommand{\dd}{{\rm d}}
\newcommand{\ii}{{\rm i}}

\renewcommand{\Re}{\mathop\mathrm{Re}\nolimits}
\renewcommand{\Im}{\mathop\mathrm{Im}\nolimits}
\newcommand{\supp}{\mathop\mathrm{supp}\nolimits}

\newcommand*{\pFqskip}{8mu} % Macro for hypergeometric series
\catcode`,\active
\newcommand*{\pFq}{\begingroup
        \catcode`\,\active
        \def ,{\mskip\pFqskip\relax}%
        \dopFq
}
\catcode`\,12
\def\dopFq#1#2#3#4#5{%
        {}_{#1}F_{#2}\kern-.1em\biggl(\kern-.1em\genfrac..{0pt}{}{#3}{#4}\biggl|\,#5\kern-.1em\biggr)%
        \endgroup
}

\newcommand{\llFrrr}[3]{\pFq{2}{3}{#1}{#2}{#3}}

\makeatother

\begin{document}

\title[]{The asymptotic zero distribution of Lommel polynomials as functions of the order with a variable complex argument}

\author{Petr Blaschke}
\address[Petr Blaschke]{Mathematical Institute in Opava, Silesian University in Opava, Na Rybn{\' i}{\v c}ku~626/1, 746 01 Opava, Czech Republic
	}
\email{Petr.Blaschke@math.slu.cz}

\author{Franti\v sek \v Stampach}
\address[Franti{\v s}ek {\v S}tampach]{
	Department of Applied Mathematics, Faculty of Information Technology, Czech Technical University in~Prague, 
	Th{\' a}kurova~9, 160~00 Praha, Czech Republic
	}	
\email{stampfra@fit.cvut.cz}

\subjclass[2010]{33C45, 26C10, 30C15, 30E15}

\keywords{Lommel polynomials, zeros, asymptotic zero distribution, variable parameters}

\date{}

\begin{abstract}
We study the asymptotic distribution of roots of Lommel polynomials as polynomials of the order with a variable and purely imaginary argument. The roots are complex and accumulate on certain curves in the complex plane. We prove existence of the weak limit of corresponding root-counting measures and deduce formulas for the supporting curves and density. The obtained result represents a solvable example of a more general problem which is still open. Numerical illustrations of the main result are also involved.
\end{abstract}

\maketitle

\section{Introduction and main results}

It is well known that the Lommel polynomials determined
by the recurrence
\begin{equation}
 R_{n+1,\nu}(x)=\frac{2(\nu+n)}{x}R_{n,\nu}(x)-R_{n-1,\nu}(x), \quad n\in\N_{0},
\label{eq:lommel_recur}
\end{equation}
with the standard initial setting $R_{-1,\nu}(x)=0$ and $R_{0,\nu}(x)=1$, for $\nu\in\C$ and $x\in\C\setminus\{0\}$, are intimately related to Bessel functions. First of all, Lommel polynomials arise explicitly in various formulas within the theory of Bessel functions, see, for example,~\cite[\S~ 9.6-9.73]{watson44} or~\cite[Chp.~VII]{erdelyi-etal81}. In addition, for $\nu>0$, $R_{n,\nu}$ are orthogonal polynomials with respect to a discrete measure supported on the set of zeros of the Bessel function of the first kind of order $\nu-1$ as explained, for instance, in~\cite{dic_54,dic-pol-wan_56}, see also \cite[Chp.~VI, \S~6]{chihara78}. Here we follow the traditional notation though, obviously, $R_{n,\nu}(x)$ is a polynomial in the variable $x^{-1}$ rather than in $x$. The variable $x$ is regarded as the argument of the $n$th Lommel polynomial while $\nu$ is referred to as the order.

Lommel polynomials can be also addressed as polynomials in the order~$\nu$. For $x\in\R\setminus\{0\}$ fixed, these polynomials are also orthogonal with respect to a measure supported this time on the set of zeros of the Bessel function of the first kind regarded as a function of the order. It was Dickinson who originally formulated the problem of constructing the measure of orthogonality for the Lommel polynomials in the variable~$\nu$ in~\cite{dic_58} which was solved by Maki in~\cite{mak_68} ten years later.

The aim of this paper is to investigate the asymptotic behavior of roots of Lommel polynomials as polynomials in the order~$\nu$, for $n\to\infty$ when the argument $x$ is also $n$-dependent and purely imaginary, namely $x=\ii n\alpha$ with $\alpha\in\R\setminus\{0\}$. Such a study belongs to the class of problems focused on the asymptotic properties of roots of orthogonal polynomials with varying non-standard parameters. Here ``non-standard parameters'' usually means that involved parameters are, for example, such that the studied polynomials are not orthogonal with respect to a \emph{positive} measure supported in~$\R$. As a result, roots of such polynomials are not necessarily real and can cluster in interesting subsets of the complex plane. Moreover, the methods for the asymptotic analysis of the polynomials with non-standard parameters and their zeros are more involved and some related general problems remain completely open (see Section~\ref{sec:appl}) in contrast to the classical theory~\cite{kui-van_99}. The number of publications devoted to these problems is growing considerably; we mention at least 
works~\cite{dia-men-ori_11, kui-mar-fin_04, kui-mcl_01, kui-mcl_04, mar-fin-etal_00, mar-fin-etal_01, mar-fin-ori_05,won-zha_06} on the asymptotic
behavior of Laguerre and Jacobi polynomials with variable non-standard parameters, a relevant research focused on hypergeometric polynomilas~\cite{aba-bog_16, aba-bog_18, dri-jor_03, dri-joh_07, dur-gui_01, sri-wan-zho_12} and other interesting polynomial families~\cite{boy-goh_07, boy-goh_08, boy-goh_10, son-won_17}. Asymptotic expansions of $R_{n,\nu}(Nx)$ for $N=n+\nu\to+\infty$ and its (real) zeros were obtained recently in~\cite{lee-wong_14}.

Typical strategies for the necessary asymptotic analysis used in the above mentioned works rely either on the saddle point method applied to a convenient integral representation of the polynomials or more advanced methods like the Riemann--Hilbert problem techniques. While the former approach is applicable when a simple generating function is known, see for example~\cite{boy-goh_10}, the latter approach can be applied if a non-Hermitian orthogonality relation is available~\cite{dea-huy-kui_10, kui-mar-fin_04, mar-fin-etal_01}. Another approach is the turning-point (or WKB) method which is applicable if the polynomials fulfill a suitable difference or differential equation, see~\cite{wan-won_16}. None of these strategies seems to be readily applicable in the case of Lommel polynomials in~$\nu$ with variable argument because, to our best knowledge, no suitable generating function, complex orthogonality, or difference/differential equation for the studied family of polynomials is known. A successful strategy applied in this article to deduce asymptotic behavior of the studied polynomials makes use of the close connection between Lommel polynomials and Bessel functions and take the advantage of known asymptotic expansions of Bessel functions in a complex setting which is mainly due to Olver~\cite{olver97}.

The explicit formula for the Lommel polynomials
\[
 R_{n,\nu}(x)=\sum_{k=1}^{\lfloor\frac{n}{2}\rfloor}(-1)^{k}\binom{n-k}{k}\left(\nu+k\right)_{n-2k}\left(\frac{2}{x}\right)^{n-2k},
\]
where $(a)_{n}=a(a+1)\dots(a+n-1)$ is the Pochhhamer symbol of $a\in\C$ and $\lfloor x\rfloor$ stands for the floor of $x\in\R$, can be rewritten in terms of the terminating hypergeometric series as
\begin{equation}
R_{n,\nu}(x)=(\nu)_{n}\left(\frac{2}{x}\right)^{n}\llFrrr{-n/2,-(n-1)/2}{-n,\nu,1-\nu-n}{-x^{2}}.
\label{eq:hypergeom_lommel}
\end{equation}
If follows easily from the hypergeometric representation that
\begin{equation}
 R_{n,\nu}(-x)=(-1)^{n}R_{n,\nu}(x)=R_{n,1-n-\nu}(x),
\label{eq:symm_lommel}
\end{equation}
for $n\in\N_{0}$, $\nu\in\C$, and $x\in\C\setminus\{0\}$. When the zeros of $R_{n,\nu}(x)$, with $x=\ii\alpha n$, are concerned, it turns out that the variable $\nu$ has to be appropriately scaled in order to keep the roots in a compact set as $n\to\infty$. Moreover, it will be advantageous to preserve symmetries which follows from~\eqref{eq:symm_lommel}. For these reasons, we introduce polynomials
\begin{equation}
 Q_{n}^{(\alpha)}(z):=R_{n,(1-n-nz)/2}(\ii\alpha n),
 \label{eq:def_Q_n}
\end{equation}
for $n\in\N_{0}$, $z\in\C$, and $\alpha\in\R\setminus\{0\}$. Then~\eqref{eq:symm_lommel} implies symmetry relations
\begin{equation}
 Q_{n}^{(\alpha)}(-z)=(-1)^{n}Q_{n}^{(\alpha)}(z) \quad\mbox{ and }\quad \overline{Q_{n}^{(\alpha)}(z)}=Q_{n}^{(\alpha)}(\overline{z}),
\label{eq:symm_Q_n}
\end{equation}
for $\alpha\in\R\setminus\{0\}$, where the bar denotes the complex conjugation. Consequently, the roots of $Q_{n}^{(\alpha)}$ are distributed symmetrically with respect to the real as well as the imaginary line for $\alpha\in\R$. In addition, since
\[
 Q_{n}^{(-\alpha)}(z)=(-1)^{n}Q_{n}^{(\alpha)}(z),
\]
we may restrict $\alpha$ to positive reals without loss of generality.

To any sequence of polynomials $P_{n}$, where the degree of $P_{n}$ equals $n$, one can associate the sequence of their root-counting measures, i.e., the probability measures
\begin{equation}
 \mu_{n}=\frac{1}{n}\sum_{k=1}^{n}\delta_{z_{k}^{(n)}},
\label{eq:root_count_meas}
\end{equation}
where $z_{1}^{(n)},\dots,z_{n}^{(n)}$ are roots of $P_{n}$ counted repeatedly according to their multiplicities and $\delta_{x}$ is the unit mass Dirac delta measure supported on the one-point set~$\{x\}$. We call the weak${}^{*}$ limit of~$\mu_{n}$, for $n\to\infty$, the \emph{asymptotic zero distribution} of the polynomial sequence $P_{n}$, provided that the limit exists. If supports of the sequence of root-counting measures remain in a compact subset of~$\C$, which is the case if $P_{n}=Q_{n}^{(\alpha)}$, the weak${}^{*}$ limit coincides with the distributional limit, i.e.,
\begin{equation}
 \lim_{n\to\infty}\int_{\C}f(z)\dd\mu_{n}(z)=\int_{\C}f(z)\dd\mu(z), \quad \forall f\in C_{0}^{\infty}(\C),
\label{eq:lim_distr_azd}
\end{equation}
for a probability measure $\mu$ supported in~$\C$. Our main result proves the existence of the asymptotic zero distribution of the sequence of polynomials $Q_{n}^{(\alpha)}$ and, in addition, provides an explicit description of its support and density.

Since the zeros of~$Q_{n}^{(\alpha)}$ are symmetrically distributed with respect to the axes we describe their asymptotic distribution in the first quadrant $\{z\in\C \mid \Re z\geq0, \Im z\geq0\}$ only. It turns out that the zeros accumulate in two arcs in the first quadrant depending on the value of $\alpha>0$. The only exception, when the zeros accumulate in a single arc, occurs when $\alpha=\alpha_{0}$, where $\alpha_{0}$ is the unique positive solution of the transcendental equation
\begin{equation}
 \sqrt{1+4\alpha^{2}}+\log\frac{2\alpha}{1+\sqrt{1+4\alpha^{2}}}=0.
\label{eq:def_alp_0}
\end{equation}
One readily checks that the left-hand side of~\eqref{eq:def_alp_0} is a strictly increasing function of $\alpha$ mapping~$(0,\infty)$ onto~$\R$. Therefore the solution~$\alpha_{0}$ exists and is unique. The numerical value of the threshold is approximately $\alpha_{0}\approx0.33137$. If not stated otherwise, complex functions such as the logarithm, the square root, etc., assume their principal branches. Our main result is as follows.

\begin{thm}\label{thm:main}
 For $\alpha>0$, the asymptotic zero distribution $\mu^{(\alpha)}$ of the sequence of polynomials $Q_{n}^{(\alpha)}$ exists and is symmetric with respect to the real and the imaginary line. Within the first quadrant, $\mu^{(\alpha)}$ decomposes as 
 \[
  \mu^{(\alpha)}\upharpoonleft_{\{\Re z\geq0, \Im z\geq0\}}=\mu_{1}^{(\alpha)}+\mu_{2}^{(\alpha)}.
 \]
The measure $\mu_{1}^{(\alpha)}$ is supported on a simple arc of a curve determined implicitly by the equation
\begin{equation}
 \Re\left[\frac{\ii\pi z}{4}+\sqrt{(1-z)^{2}+4\alpha^{2}}-\frac{1-z}{2}\log\left(\frac{1}{2\alpha}\left(1-z+\sqrt{(1-z)^{2}+4\alpha^{2}}\right)\right)\right]=0,
\label{eq:impl_supp_mu_1}
\end{equation}
for $z\in[0,1]+\ii[0,2\alpha]$. This arc connects a unique point $\xi(\alpha)\in[0,1)\cup\ii[0,2\alpha)$ with the point $1+2\ii\alpha$. The measure $\mu_{1}^{(\alpha)}$ is absolutely continuous on the arc and its density reads
\begin{equation}
 \frac{\dd\mu_{1}^{(\alpha)}}{\dd z}(z)=\frac{1}{4}+\frac{1}{2\pi\ii}\log\left(\frac{1}{2\alpha}\left(1-z+\sqrt{(1-z)^{2}+4\alpha^{2}}\right)\right),
 \label{eq:mu_1_dens}
\end{equation}
where $z$ transverses the supporting curve from the point~$\xi(\alpha)$ towards~$1+2\alpha\ii$. 
For the measure~$\mu_{2}^{(\alpha)}$, we have the following possibilities:
\begin{enumerate}[{\upshape i)}]
 \item If $\alpha<\alpha_{0}$, then $\xi(\alpha)\in(0,1)$, $\supp\mu^{(\alpha)}_{2}=[0,\xi(\alpha)]$, and
 \[
  \frac{\dd\mu_{2}^{(\alpha)}}{\dd x}(x)=\frac{1}{2}, \quad\mbox{ for }\; 0\leq x \leq\xi(\alpha).
 \]
 \item If $\alpha>\alpha_{0}$, then $\xi(\alpha)\in\ii(0,2\alpha)$, $\supp\mu^{(\alpha)}_{2}=[0,\Im\xi(\alpha)]$, and
 \[
  \frac{\dd\mu_{2}^{(\alpha)}}{\dd y}(y)=\frac{1}{\pi}\log\left(\frac{1}{2\alpha}\left|1+\ii y+\sqrt{(1+\ii y)^{2}+4\alpha^{2}}\right|\right), \quad\mbox{ for }\; 0\leq y \leq\Im\xi(\alpha).
 \]
 \item If $\alpha=\alpha_{0}$, then $\xi(\alpha)=0$ and $\mu^{(\alpha)}_{2}=0$.
\end{enumerate}
\end{thm}

It is worth noticing that, if $\alpha<\alpha_{0}$, a portion of zeros of~$Q_{n}^{(\alpha)}$ is asymptotically uniformly distributed in a real interval, while the uniformity is not preserved on the imaginary line segment in the case $\alpha>\alpha_{0}$.

The paper is organized as follows. In Section~\ref{sec:prelim}, a general method and certain selected formulas for Lommel polynomials and Bessel functions are recalled as preliminaries. The proof of Theorem~\ref{thm:main} is worked out in Section~\ref{sec:main} in several steps. Numerical illustrations of Theorem~\ref{thm:main} are presented in Section~\ref{sec:numer}. Finally, in Section~\ref{sec:appl}, we show that Theorem~\ref{thm:main} can be viewed as a solvable model for a more general and still open problem concerning the asymptotic eigenvalue distribution of a sequence of non-self-adjoint sampling Jacobi matrices which served as a partial motivation for the current article.

\section{Preliminaries}\label{sec:prelim}

\subsection{The method}

The method for a derivation of the asymptotic zero distribution, which we will use, relies on finding a formula for a limiting Cauchy transform and inverting. Although this is a quite well known approach, we briefly describe its main steps using the theory of generalized functions following primarily the book~\cite{vladimirov84}. The reader may consult also~\cite{hormanderI,saff-totik97}.

Recall the Cauchy (or Stieltjes) transform of a Borel measure~$\mu$ is defined as
\[
 C_{\mu}(z):=\int_{\C}\frac{\dd\mu(\xi)}{z-\xi}, \quad z\in\C\setminus\supp\mu.
\]
In terms of the generalized functions, the Cauchy transform can be viewed as an element of the space of generalized functions~$\mathscr{D}'(\C)$ given by the convolution
\begin{equation}
 C_{\mu}(z)=\frac{1}{z}\ast\mu \;\mbox{ in } \mathscr{D}'(\C).
\label{eq:cauchy_generalized}
\end{equation}
Recall that $z^{-1}\in L_{loc}^{1}(\C)$ and hence the convolution kernel $z^{-1}$ is a regular generalized function. Moreover, we always assume that $\mu$ is compactly supported, in which case the existence of the convolution~\eqref{eq:cauchy_generalized} is guaranteed in the generalized sense. Recall also that $C_{\mu}(z)$ as a complex function of $z$ is analytic in $\C\setminus\supp\mu$. Conversely, if $C_{\mu}$ is analytic in a region $\Omega\subset\C$, then $\supp\mu\cap\Omega=\emptyset$.

Since $\pi^{-1}z^{-1}$ is a fundamental solution for the Cauchy--Riemann operator, i.e.,
\[
 \partial_{\overline{z}}\,\frac{1}{\pi z}=\delta(z)\;\mbox{ in } \mathscr{D}'(\C),
\]
where $\partial_{\bar{z}}=(\partial_{x}+\ii\partial_{y})/2$, for $z=x+\ii y$, the equation~\eqref{eq:cauchy_generalized} can be easily inverted getting
\begin{equation}
 \mu=\frac{1}{\pi}\partial_{\overline{z}}\,C_{\mu}(z) \;\mbox{ in } \mathscr{D}'(\C).
\label{eq:mu_cauchy_generalized}
\end{equation}
Suppose, in addition, that $C_{\mu}$ is an analytic function everywhere except several branch cuts occurring on a finite number of smooth closed arcs. Then, if $\gamma:(a,b)\to\C$ denotes a simple oriented open smooth curve on which $C_{\mu}$ has the branch cut then $\mu$ is absolutely continuous on the curve and its density reads
\begin{equation}
\frac{\dd\mu}{\dd x}(x)=-\frac{\gamma'(x)}{2\pi\ii}\left(C_{\mu}(\gamma(x)+)-C_{\mu}(\gamma(x)-)\right), \quad x\in(a,b),
\label{eq:plemelj-sokhotski}
\end{equation}
provided that the jump of $C_{\mu}$ on $\gamma$ is Lebesgue integrable. Here $C_{\mu}(\gamma(x)\pm)$ denote the non-tangential limits from the left/right side of $\gamma$ induced by the chosen orientation. The formula~\eqref{eq:plemelj-sokhotski} follows from the application of a formula for generalized derivatives, see~\cite[Sec.~6.5]{vladimirov84}, to~\eqref{eq:mu_cauchy_generalized}. If the brach cut of $C_{\mu}$ occurs on a real interval, the formula~\eqref{eq:plemelj-sokhotski} is known as the Plemelj--Sokhotski formula.

For a polynomial $P_{n}$ of degree $n$ with roots $z_{1}^{(n)},\dots,z_{n}^{(n)}$, 
the Cauchy transform of the root-counting measure~\eqref{eq:root_count_meas} reads
\begin{equation}
 C_{\mu_{n}}(z)=\frac{P'_{n}(z)}{nP_{n}(z)},
\label{eq:cauchy_root_meas}
\end{equation}
 for $z\in\C\setminus\{z_{1}^{(n)},\dots,z_{n}^{(n)}\}$. 
 Recall that we suppose the roots $z_{1}^{(n)},\dots,z_{n}^{(n)}$ remain in a fixed compact subset of~$\C$ for all $n\in\N$. The approach, which we will use, comprises the following general steps. First, we derive the leading term of the asymptotic expansion of $P_{n}$ for $n\to\infty$ which we use in~\eqref{eq:cauchy_root_meas} to find a limit of~$C_{\mu_{n}}(z)$,
 for $n\to\infty$, in the generalized sense. To this end, it is sufficient to show that
 there exists a limiting function, say $C$, such $C_{\mu_{n}}$ converges to $C$, as $n\to\infty$, point-wise almost everywhere in $\C$ (with respect to the 2D~Lebesgue measure).
 Using~\eqref{eq:mu_cauchy_generalized}, one concludes that $C=C_{\mu}$, where $\mu$ is the distributional limit of $\mu_{n}$ for $n\to\infty$. Finally, it turns out that the limiting Cauchy transform $C_{\mu}$ is analytic up to several branch cuts occurring on certain curves in~$\C$ and the inversion formula~\eqref{eq:plemelj-sokhotski} applies.

\subsection{Preliminary formulas for Lommel polynomials and Bessel functions}

Following the traditional notation, we denote by $J_{\nu}$, $I_{\nu}$, and $K_{\nu}$ the standard branches of the Bessel function of the first kind, the modified Bessel function of the first kind, and the modified Bessel function of the second kind, respectively. Recall that, if $\nu$ is not an integer, these functions are analytic in $\C\setminus(-\infty,0]$ with branch cuts~$(-\infty,0]$; see~\cite{watson44} or~\cite[Chp.~10]{dlmf} for more details.

First, we derive a formula for Lommel polynomials with purely complex argument in terms of the modified Bessel functions which is in a suitable form for a subsequent asymptotic analysis of polynomials~$Q_{n}^{(\alpha)}$ for $n\to\infty$.

\begin{lem}\label{lem:lommel_mod_bessel}
 For all $n\in\N_{0}$, $\nu\in\C$, $z\in\C\setminus(-\infty,0]$, one has
 \begin{align*}
  R_{n,\nu}(\ii z)=\ii^{n}z\bigg[I_{n+\nu}(z)K_{1-\nu}(z)+(-1)^{n}K_{n+\nu}(z)&I_{1-\nu}(z)\\
  &+(-1)^{n}\frac{2\sin(\pi\nu)}{\pi}K_{n+\nu}(z)K_{1-\nu}(z)\bigg].
 \end{align*}
\end{lem}

\begin{proof}
We start with the well-known formula~\cite[\S~9.61, Eq.~(2)]{watson44}
\begin{equation}
 R_{n,\nu}(z)=\frac{\pi z}{2\sin(\pi\nu)}\left(J_{\nu+n}(z)J_{-\nu+1}(z)+(-1)^{n}J_{-\nu-n}(z)J_{\nu-1}(z)\right)
\label{eq:lommel_bes_J}
\end{equation}
which holds true for $n\in\N_{0}$, $z\in\C\setminus(-\infty,0]$, and $\nu\in\C\setminus\Z$. The validity of~\eqref{eq:lommel_bes_J} extends to integer values of $\nu$, too, by taking the respective limit on the right-hand side.
Then the identity~\cite[Eq.~10.27.6]{dlmf}
\[
 I_{\nu}(z)=e^{-\ii\pi\nu/2}J_{\nu}(\ii z)
\]
used in~\eqref{eq:lommel_bes_J} yields
\begin{equation}
 R_{n,\nu}(\ii z)=-\frac{\ii^{n}\pi z}{2\sin(\pi\nu)}\left(I_{\nu+n}(z)I_{-\nu+1}(z)-I_{-\nu-n}(z)I_{\nu-1}(z)\right).
\label{eq:lommel_bes_I_inproof}
\end{equation}
Now it suffices to apply the connection formula~\cite[Eq.~10.27.2]{dlmf} 
\[
 I_{-\nu}(z)=I_{\nu}(z)+\frac{2\sin(\pi\nu)}{\pi}K_{\nu}(z)
\]
to $I_{-n-\nu}(z)$ and $I_{\nu-1}(z)$ in~\eqref{eq:lommel_bes_I_inproof} which results in the identity from the statement.
\end{proof}

Second, we recall asymptotic expansions of the modified Bessel function for large argument and order.

\begin{lem}[Olver]\label{lem:olver}
 For $\nu\to\infty$ in the half-plane $\Re\nu>0$, we have
 \begin{align}
  I_{\nu}(\nu z)&=\frac{1}{\sqrt{2\pi\nu}}\frac{e^{\nu\zeta(z)}}{\left(1+z^{2}\right)^{1/4}}\left(1+O\left(\frac{1}{\nu}\right)\right),\label{eq:asympt_I1}\\
  K_{\nu}(\nu z)&=\sqrt{\frac{\pi}{2\nu}}\frac{e^{-\nu\zeta(z)}}{\left(1+z^{2}\right)^{1/4}}\left(1+O\left(\frac{1}{\nu}\right)\right),\label{eq:asympt_K1}\\
  I_{\nu+1}(\nu z)&=\frac{1}{\sqrt{2\pi\nu}}\frac{\sqrt{1+z^{2}}-1}{z\left(1+z^{2}\right)^{1/4}}\,e^{\nu\zeta(z)}\left(1+O\left(\frac{1}{\nu}\right)\right),\label{eq:asympt_I2}\\
  K_{\nu+1}(\nu z)&=\sqrt{\frac{\pi}{2\nu}}\frac{\sqrt{1+z^{2}}+1}{z\left(1+z^{2}\right)^{1/4}}\,e^{-\nu\zeta(z)}\left(1+O\left(\frac{1}{\nu}\right)\right),\label{eq:asympt_K2}
 \end{align}
 uniformly in the sector $|\arg z|<\pi/2-\epsilon$, for arbitrary $0<\epsilon<\pi/2$, where
 \begin{equation}
  \zeta(z):=\sqrt{1+z^{2}}+\log\frac{z}{1+\sqrt{1+z^{2}}}.
 \label{eq:def_zeta}
 \end{equation}
\end{lem}

\begin{proof}
  Expansions~\eqref{eq:asympt_I1} and~\eqref{eq:asympt_K1} were proved by Olver in~\cite{olver54}. Due to the uniformity of these expansion, we can differentiate~\eqref{eq:asympt_I1} and~\eqref{eq:asympt_K1} with respect to~$z$ getting asymptotic formulas
  \begin{equation}
   I_{\nu}'(\nu z)=\frac{1}{\sqrt{2\pi\nu}}\frac{\zeta'(z)e^{\nu\zeta(z)}}{\left(1+z^{2}\right)^{1/4}}\left(1+O\left(\frac{1}{\nu}\right)\right)
  \label{eq:asympt_I1der}
  \end{equation}
  and
  \begin{equation}
   K_{\nu}'(\nu z)=-\sqrt{\frac{\pi}{2\nu}}\frac{\zeta'(z)e^{-\nu\zeta(z)}}{\left(1+z^{2}\right)^{1/4}}\left(1+O\left(\frac{1}{\nu}\right)\right),
  \label{eq:asympt_K1der}
  \end{equation}
  for $\nu\to\infty$ in the half-plane $\Re\nu>0$. 
  Noticing that
  \begin{equation}
   \zeta'(z)=\frac{\sqrt{1+z^{2}}}{z}
   \label{eq:zeta_dif}
  \end{equation}
  and using the identities~\cite[Eq.~10.29.2]{dlmf}
  \[
   I_{\nu+1}(\nu z)=I_{\nu}'(\nu z)-\frac{1}{z}I_{\nu}(\nu z)
  \quad \mbox{  and }\quad 
   K_{\nu+1}(\nu z)=-K_{\nu}'(\nu z)+\frac{1}{z}K_{\nu}(\nu z)
  \]
  together with the expansions \eqref{eq:asympt_I1},\eqref{eq:asympt_I1der} and \eqref{eq:asympt_K1},\eqref{eq:asympt_K1der}, we arrive at the asymptotic formulas~\eqref{eq:asympt_I2} and~\eqref{eq:asympt_K2}.
\end{proof}

\section{Proof of the main result}\label{sec:main}

The aim of this section is to prove Theorem~\ref{thm:main}. This is done in several steps.

\subsection{A rough localization of roots}

First, we determine a compact subset of~$\C$ where all roots of $Q_{n}^{(\alpha)}$ are located for all $n\in\N$. Such a localization allows us to restrict the forthcoming analysis to the compact subset which simplifies the overall derivation of the asymptotic zero distribution.

\begin{lem}\label{lem:local}
 Let $\alpha>0$ and $n\in\N$. If $Q_{n}^{(\alpha)}(z_{0})=0$, then 
 \[
  |\Re z_{0}|\leq 1-\frac{1}{n} \quad\mbox{ and }\quad |\Im z_{0}|\leq 2\alpha\cos\left(\frac{\pi}{n+1}\right).
 \]
 Consequently, all zeros are located in a rectangular domain:
 \[
 \bigcup_{n\in\N}\left\{z\in\C \;\Big|\; Q_{n}^{(\alpha)}(z)=0\right\}\subset(-1,1)+2\alpha\ii(-1,1).
 \]
\end{lem}

\begin{proof}
 First, with the aid of the recurrence~\eqref{eq:lommel_recur} and definition~\eqref{eq:def_Q_n}, one readily verifies that
 \[
  (\ii\alpha)^{n}Q_{n}^{(\alpha)}(z)=\det(J_{n}-z),
 \]
 where $J_{n}\in\C^{n,n}$ is a tridiagonal matrix with entries
 \[
  \left(J_{n}\right)_{k,k}=-1+\frac{2k-1}{n}, \quad k=1,\dots,n,
 \]
 and
 \[
  \left(J_{n}\right)_{k,k+1}=\left(J_{n}\right)_{k+1,k}=\ii\alpha, \quad k=1,\dots,n-1.
 \]
 Hence roots of $Q_{n}^{(\alpha)}$ coincide with eigenvalues of $J_{n}$.
 
 Let $Q_{n}^{(\alpha)}(z_{0})=0$. Denote by $\phi\in\C^{n}$ a normalized eigenvector of $J_{n}$ corresponding to the eigenvalue~$z_{0}$. Noticing that $\Re J_{n}=(J_{n}+J_{n}^{*})/2$ is the diagonal matrix with the same diagonal as $J_{n}$, one obtains
 \[
  |\Re z_{0}|=|\langle\phi,(\Re J_{n})\phi\rangle|\leq\sum_{k=1}^{n}\frac{|n-2k+1|}{n}|\phi_{k}|^{2}\leq 1-\frac{1}{n}.
 \]
 
 Similarly, one shows that $|\Im z_{0}|$ is majorized by the largest (in modulus) eigenvalue of the Hermitian matrix $\Im J_{n}=(J_{n}-J_{n}^{*})/2\ii$. It is a matter of simple linear algebra to show that the eigenvalues of $\Im J_{n}$ are
 \[
  2\alpha\cos\left(\frac{\pi k}{n+1}\right), \quad k=1,\dots,n.
 \]
 Consequently, $|\Im z_{0}|\leq 2\alpha\cos(\pi/(n+1))$.
\end{proof}

\subsection{Asymptotic behavior of $Q_{n}^{(\alpha)}$}

Due to symmetries~\eqref{eq:symm_Q_n} and Lemma~\ref{lem:local} the analysis of the zeros of $Q_{n}^{(\alpha)}$ can be restricted to a rectangular domain of the first quadrant
which we denote by
\begin{equation}
 \Omega^{(\alpha)}:=(0,1)+2\ii\alpha(0,1),
\end{equation}
for later purposes. Further, we will use the auxiliary function
\begin{equation}
 \chi_{\alpha}(z):=\frac{1+z}{2}\zeta\left(\frac{2\alpha}{1+z}\right)=\frac{1}{2}\sqrt{4\alpha^{2}+(1+z)^{2}}+\frac{1+z}{2}\log\frac{2\alpha}{1+z+\sqrt{4\alpha^{2}+(1+z)^{2}}},
\label{eq:def_chi}
\end{equation}
where $\zeta$ is as in~\eqref{eq:def_zeta} and $\alpha>0$. By differentiating~\eqref{eq:def_chi} with respect to~$z$, one obtains the formula
 \begin{equation}
  \chi_{\alpha}'(z)=\frac{1}{2}\zeta\left(\frac{2\alpha}{1+z}\right)-\frac{\alpha}{1+z}\zeta'\left(\frac{2\alpha}{1+z}\right)=\frac{1}{2}\log\frac{2\alpha}{1+z+\sqrt{(1+z)^{2}+4\alpha^{2}}}
 \label{eq:chi_dif}
 \end{equation}
which will be also used multiple times below. Next, we will need the following auxiliary inequality.

\begin{lem}\label{lem:non-posit}
 For all $\alpha>0$ and $z\in\Omega^{(\alpha)}$, one has
  $\Re \chi_{\alpha}(z)<\Re \chi_{\alpha}(-z)$.
\end{lem}

\begin{proof}
 For $\alpha>0$ fixed, we temporarily denote $f(z):=\Re\left(\chi_{\alpha}(z)-\chi_{\alpha}(-z)\right)$. Since $\chi_{\alpha}$ is analytic in $\Omega^{(\alpha)}$ as well as in $-\Omega^{(\alpha)}$, $f$ is harmonic in~$\Omega^{(\alpha)}$ and continuous to the boundary of~$\Omega^{(\alpha)}$. Since non-constant, a maximum of $f$ is attained on the boundary of~$\Omega^{(\alpha)}$ according to the Maximum Modulus Principle for harmonic functions. Hence it suffices to verify that~$f$ is non-positive on the boundary of~$\Omega^{(\alpha)}$.
 
 Note that the function~$\zeta$ given by~\eqref{eq:zeta_dif} fulfills $\overline{\zeta(z)}=\zeta(\overline{z})$. Therefore it follows immediately from~\eqref{eq:def_chi} that 
 \[
 f(\ii y)=\Re\left(\chi_{\alpha}(\ii y)-\overline{\chi_{\alpha}(\ii y)}\right)=0,
 \]
 for $y\in[0,2\alpha]$. 
 
 Using~\eqref{eq:chi_dif}, it is elementary to show that $\chi_{\alpha}'(x)<0$ for all $x\in(-1,1)$. Thus
 \[
  f'(x)=\chi_{\alpha}'(x)+\chi_{\alpha}'(-x)<0,
 \]
 for $x\in(0,1)$ which means that $f$ is strictly decreasing in $(0,1)$. Taking into account that $f(0)=0$, one concludes that $f(x)<0$ for~$x\in(0,1]$.
 
 Next, one can differentiate~\eqref{eq:chi_dif} once more getting the simple expression
 \[
  \chi_{\alpha}''(z)=-\frac{1}{2\sqrt{(1+z)^{2}+4\alpha^{2}}}.
 \]
 Consequently, $\Re\chi_{\alpha}''(z)<0$ for any $z\in\C$ which is not of the form $z=-1+\ii y$ with $|y|\geq2\alpha$. Noticing additionally that $\Re\chi_{\alpha}'(-1+2\alpha\ii)=0$, one concludes that $\Re\chi_{\alpha}'(x+2\alpha\ii)<0$ for $x\in(-1,1)$ and hence $f'(x+2\alpha\ii)<0$ for $x\in(0,1)$. Since $f(2\alpha\ii)=0$, we see that $f(x+2\alpha\ii)<0$ for $x\in(0,1]$.
 
 A completely analogous reasoning as in the previous step shows that $f(1+\ii y)$ is a decreasing function of $y\in(0,2\alpha)$. It already follows from the previous analysis that $f(1)<0$
 and so $f(1+\ii y)<0$ for $y\in[0,1]$. In total, the function $f$ assumes only non-positive values on the boundary of~$\Omega^{(\alpha)}$ which completes the proof.
\end{proof}

As a next step, we investigate an asymptotic behavior of the auxiliary polynomials
\begin{equation}
 q_{n}^{(\alpha)}(z):=\frac{\ii^{n}}{\alpha n}\,Q_{n}^{(\alpha)}\left(z-\frac{1}{n}\right).
\label{eq:def_q_n}
\end{equation}
The shift of the argument in~\eqref{eq:def_q_n} allows a straightforward application of the asymptotic expansions from Lemma~\ref{lem:olver}.

\begin{lem}\label{lem:q_n_asympt}
 For $\alpha>0$, one has
 \[
  q_{n}^{(\alpha)}(z)=B_{n}^{(\alpha)}(z)+C_{n}^{(\alpha)}(z),
 \]
 where
 \begin{equation}
  B_{n}^{(\alpha)}(z)=\frac{(-1)^{n}}{n}g^{(\alpha)}_{-}(z)e^{-n\left(\chi_{\alpha}(z)-\chi_{\alpha}(-z)\right)}\left(1+O\left(\frac{1}{n}\right)\right), \quad n\to\infty,
  \label{eq:B_n_asympt}
 \end{equation}
 \begin{equation}
  C_{n}^{(\alpha)}(z)=\frac{(-1)^{n}\ii^{n+1}}{n}g^{(\alpha)}_{+}(z)e^{-n\left(\chi_{\alpha}(z)+\chi_{\alpha}(-z)+\ii\pi z/2\right)}\left(1+O\left(\frac{1}{n}\right)\right), \quad n\to\infty,
  \label{eq:C_n_asympt}
 \end{equation}
 and
 \begin{equation}
  g^{(\alpha)}_{\pm}(z):=\frac{1}{2\alpha}\frac{\sqrt{(1-z)^{2}+4\alpha^{2}}\pm(1-z)}{\left((1-z)^{2}+4\alpha^{2}\right)^{1/4}\left((1+z)^{2}+4\alpha^{2}\right)^{1/4}}.
 \label{eq:def_g}
 \end{equation}
 In addition, the asymptotic formulas~\eqref{eq:B_n_asympt} and~\eqref{eq:C_n_asympt} are locally uniform in $z\in\Omega^{(\alpha)}$.
\end{lem}

\begin{proof}
 It follows from~\eqref{eq:def_Q_n},\eqref{eq:def_q_n} and Lemma~\ref{lem:lommel_mod_bessel} that $q_{n}^{(\alpha)}$ can be written as
 \[
  q_{n}^{(\alpha)}(z)=A_{n}^{(\alpha)}(z)+B_{n}^{(\alpha)}(z)+C_{n}^{(\alpha)}(z),
 \]
 for $n\in\N$ and $z\in\Omega^{(\alpha)}$, where
 \[
  A_{n}^{(\alpha)}(z):=I_{n(1+z)/2}(\alpha n)K_{1+n(1-z)/2}(\alpha n),
 \]
 \[
  B_{n}^{(\alpha)}(z):=(-1)^{n}K_{n(1+z)/2}(\alpha n)I_{1+n(1-z)/2}(\alpha n),
 \]
 and
 \[
  C_{n}^{(\alpha)}(z):=\frac{2}{\pi}\sin\left(\frac{\pi n}{2}(1+z)\right)K_{n(1+z)/2}(\alpha n)K_{1+n(1-z)/2}(\alpha n).
 \]

 Note that, if $z\in\Omega^{(\alpha)}$, then
 \[
  \Re \frac{n}{2}(1\pm z)>0 \quad\mbox{ and }\quad \Re\frac{2\alpha}{1\pm z}>0.
 \]
 Therefore Lemma~\ref{lem:olver} can be applied to deduce asymptotic formulas for $A_{n}^{(\alpha)}(z)$, $B_{n}^{(\alpha)}(z)$, and $C_{n}^{(\alpha)}(z)$, for $z\in\Omega^{(\alpha)}$ as $n\to\infty$. In the case of $C_{n}^{(\alpha)}(z)$, we also take into account that
 \[
  \sin\left(\frac{\pi n}{2}(1+z)\right)=\frac{\ii}{2}e^{-\ii\pi n(1+z)/2}\left(1+O\left(\frac{1}{n}\right)\right), \quad n\to\infty,
 \]
 locally uniformly in the half-plane $\Im z>0$. The resulting formulas are
 \[
  A_{n}^{(\alpha)}(z)=\frac{1}{n}g^{(\alpha)}_{+}(z)e^{n\left(\chi_{\alpha}(z)-\chi_{\alpha}(-z)\right)}\left(1+O\left(\frac{1}{n}\right)\right), \quad n\to\infty,
 \]
 and~\eqref{eq:B_n_asympt} and~\eqref{eq:C_n_asympt}. Moreover, the expansions are locally uniform in $z\in\Omega^{(\alpha)}$. 
 
 Finally, according to Lemma~\ref{lem:non-posit}, the sequence $A_{n}^{(\alpha)}(z)$ decays exponentially faster in comparison with~$B_{n}^{(\alpha)}(z)$, for $n\to\infty$ and $z\in\Omega^{(\alpha)}$. Consequently, the sequence~$A_{n}^{(\alpha)}(z)$ does not contribute to the leading term of the asymptotic expansion of~$q_{n}^{(\alpha)}(z)$ for $z\in\Omega^{(\alpha)}$ as $n\to\infty$. This justifies the claim as stated.
\end{proof}

It follows from Lemma~\ref{lem:q_n_asympt} that, for a generic $z\in\Omega^{(\alpha)}$, the asymptotic behavior of $q_{n}^{(\alpha)}(z)$ is determined either by the term $B_{n}^{(\alpha)}(z)$ or $C_{n}^{(\alpha)}(z)$, for $n\to\infty$. The latter depends on the values of the real parts of functions occurring as arguments of the exponential functions in formulas~\eqref{eq:B_n_asympt} and~\eqref{eq:C_n_asympt}. The coincidence of these values means that
\[
 \Re\left(\chi_{\alpha}(z)-\chi_{\alpha}(-z)\right)=\Re\left(\chi_{\alpha}(z)+\chi_{\alpha}(-z)+\frac{\ii\pi z}{2}\right)
\]
which happens if and only if 
\begin{equation}
 \Re \chi_{\alpha}(-z)=\frac{\pi}{4}\Im z.
\label{eq:stokes_line}
\end{equation}
It is useful to introduce the following subsets of $\Omega^{(\alpha)}$:
\[
\Omega_{\pm}^{(\alpha)}:=\left\{z\in\Omega^{(\alpha)} \;\bigg|\; \Re \chi_{\alpha}(-z)\lessgtr\frac{\pi}{4}\Im z\right\}.
\]
It turns out that $\Omega_{\pm}^{(\alpha)}\neq\emptyset$ (it follows from Lemma~\ref{lem:stokes_line_prop} below) and their common boundary given by the equation~\eqref{eq:stokes_line} determines a Stokes line of $q_{n}^{(\alpha)}$, for $n\to\infty$, in the region~$\Omega^{(\alpha)}$. A similar change of the character of asymptotic behavior exhibits the original sequence of polynomials~$Q_{n}^{(\alpha)}$ as shown in the next statement.

\begin{prop}\label{prop:Q_n_asympt}
 For $\alpha>0$ and $n\to\infty$, we have the locally uniform asymptotic expansions
 \begin{equation}
  Q_{n}^{(\alpha)}(z)=(-1)^{n}f_{+}^{(\alpha)}(z)e^{-n\left(\chi_{\alpha}(z)+\chi_{\alpha}(-z)+\ii\pi z/2\right)}\left(1+O\left(\frac{1}{n}\right)\right), \quad\mbox{ for } z\in\Omega_{+}^{(\alpha)},
  \label{eq:Q_n_asympt_plus}
 \end{equation}
 and
 \begin{equation}
  Q_{n}^{(\alpha)}(z)=\ii^{n}f_{-}^{(\alpha)}(z)e^{-n\left(\chi_{\alpha}(z)-\chi_{\alpha}(-z)\right)}\left(1+O\left(\frac{1}{n}\right)\right), \quad\mbox{ for } z\in\Omega_{-}^{(\alpha)},
  \label{eq:Q_n_asympt_minus}
 \end{equation}
 where
 \begin{equation}
  f_{\pm}^{(\alpha)}(z):=\frac{1}{2}\frac{\left(\sqrt{(1-z)^{2}+4\alpha^{2}}\pm(1-z)\right)^{\!1/2}\left(\sqrt{(1+z)^{2}+4\alpha^{2}}+1+z\right)^{\!1/2}}{\left((1-z)^{2}+4\alpha^{2}\right)^{1/4}\left((1+z)^{2}+4\alpha^{2}\right)^{1/4}}.
 \label{eq:def_f_plusminus}
 \end{equation}
\end{prop}

\begin{proof}
 Let $z\in\Omega_{+}^{(\alpha)}$. Then the leading term of the asymptotic expansion of~$q_{n}^{(\alpha)}(z)$ is determined by $C_{n}^{(\alpha)}(z)$ from Lemma~\ref{lem:q_n_asympt} and the corresponding asymptotic formula reads
 \begin{equation}
  q_{n}^{(\alpha)}(z)=\frac{(-1)^{n}\ii^{n+1}}{n}g^{(\alpha)}_{+}(z)e^{-n\phi_{\alpha}(z)}\left(1+O\left(\frac{1}{n}\right)\right), \quad n\to\infty,
 \label{eq:q_n_asympt_by_C}
 \end{equation}
 where we temporarily denote $\phi_{\alpha}(z):=\chi_{\alpha}(z)+\chi_{\alpha}(-z)+\ii\pi z/2$.
 Moreover, the asymptotic expansion~\eqref{eq:q_n_asympt_by_C} is locally uniform in $z\in\Omega_{+}^{(\alpha)}$. Using~\eqref{eq:def_q_n} and the uniformity of the expansion~\eqref{eq:q_n_asympt_by_C}, one deduces
 \[
  Q_{n}^{(\alpha)}(z)=(-1)^{n}f_{+}^{(\alpha)}(z)e^{-n\phi_{\alpha}(z)}\left(1+O\left(\frac{1}{n}\right)\right), \quad n\to\infty,
 \]
 for 
 \begin{equation}
  f_{+}^{(\alpha)}(z)=\ii\alpha g^{(\alpha)}_{+}(z)e^{-\phi_{\alpha}'(z)}.
 \label{eq:f_plus_inproof}
 \end{equation}
 Taking into account~\eqref{eq:chi_dif}, one easily computes
 \[
  e^{-\phi_{\alpha}'(z)}=-\ii\frac{\left(1+z+\sqrt{(1+z)^{2}+4\alpha^{2}}\right)^{1/2}}{\left(1-z+\sqrt{(1-z)^{2}+4\alpha^{2}}\right)^{1/2}}.
 \]
 Thus, the equation~\eqref{eq:f_plus_inproof} together with~\eqref{eq:def_g} yields the explicit expression~\eqref{eq:def_f_plusminus} for $f_{+}^{(\alpha)}$.
 
 The second asymptotic expansion~\eqref{eq:Q_n_asympt_minus} is to be deduced analogously from Lemma~\ref{lem:q_n_asympt} for $z\in\Omega_{-}^{(\alpha)}$.
\end{proof}

With the aid of~\eqref{eq:chi_dif}, one immediately verifies the following corollary of Proposition~\ref{prop:Q_n_asympt}.

\begin{cor}\label{cor:limit_cauchy}
 For $\alpha>0$, one has the limit
 \begin{equation}
  \lim_{n\to\infty}\frac{\partial_{z}Q_{n}^{(\alpha)}(z)}{nQ_{n}^{(\alpha)}(z)}\\
  =\begin{cases}-\frac{\ii\pi}{2}+h_{\alpha}(z)-h_{\alpha}(-z),&  \mbox{ for } z\in\Omega_{+}^{(\alpha)}, \\
  -\log(2\alpha)+h_{\alpha}(z)+h_{\alpha}(-z),&\mbox{ for } z\in\Omega_{-}^{(\alpha)}, \\  
   \end{cases}
 \end{equation}
 where
 \begin{equation}
  h_{\alpha}(z):=\frac{1}{2}\log\left(1+z+\sqrt{(1+z)^{2}+4\alpha^{2}}\right).
  \label{eq:def_h}
 \end{equation}
 Moreover, the convergence is locally uniform in $z\in\Omega_{\pm}^{(\alpha)}$.
\end{cor}

 The last lemma of this subsection provides information about properties of the curve defined by the equation~\eqref{eq:stokes_line} in~$\Omega^{(\alpha)}$. Recall the definition of~$\alpha_{0}$ as the unique positive solution of~\eqref{eq:def_alp_0}.

\begin{lem}\label{lem:stokes_line_prop}
 The equation~\eqref{eq:stokes_line} determines a simple smooth curve in~$\Omega^{(\alpha)}$ with one end-point at $1+2\ii\alpha$ and the second end-point located
 in $(0,1)$, if $0<\alpha<\alpha_{0}$; at the origin, if $\alpha=\alpha_{0}$; and in $\ii(0,2\alpha)$, if $\alpha>\alpha_{0}$.
\end{lem}

\begin{proof}
 We denote $f(z):=\chi_{\alpha}(-z)+\pi\ii z/4$ only for the needs of this proof.
 Then the curve is determined by the equation $\Re f(z)=0$ for $z\in\Omega^{(\alpha)}$. First, the smoothness of the curve follows from the analyticity of $\chi_{\alpha}$ in $-\Omega^{(\alpha)}$. In addition, the analyticity of $f$ implies that $\Re f$ is harmonic in~$\Omega^{(\alpha)}$. Then level curve defined by the $\Re f(z)=0$ for $z\in\Omega^{(\alpha)}$ cannot contain a loop. Otherwise the Maximum Modulus Principle together with the Open Mapping Theorem would imply that $f$ is constant in~$\Omega^{(\alpha)}$ which is not the case.
 
 Next, it is straightforward to check that $\Re f(1+2\alpha\ii)=0$ and hence $1+2\alpha\ii$ is one of the end-points.

 It remains to verify the statements for the second end-point which is to be done by inspection of the values of $\Re f$ on the boundary of~$\Omega^{(\alpha)}$. The analysis requires several computations but is not difficult. Therefore we only indicate it. For example, for the values of $f$ on the right side of the rectangle~$\Omega^{(\alpha)}$, one finds
 \[
  \Re f(1+2\alpha\ii\sin\omega)=\alpha\cos\omega+\alpha\left(\omega-\frac{\pi}{2}\right)\sin\omega,
 \]
 for $\omega\in[0,\pi/2]$, which is a strictly decreasing function of $\omega$. The only point at which it vanishes is $\omega=\pi/2$ with corresponds to the end-point $1+2\alpha\ii$.
 
 Similarly, referring to~\eqref{eq:chi_dif}, one checks that
 \[
  \Re f'(x+2\alpha\ii)=-\frac{1}{2}\log\frac{2\alpha}{|1-x-2\alpha\ii+\sqrt{(1-x)(1-x-4\alpha\ii)}|}\neq0, \quad \forall x\in(0,1),
 \]
 meaning that the function $\Re f$ restricted to the upper side of the rectangle~$\Omega^{(\alpha)}$ is either strictly increasing or strictly decreasing. In any case, the point $1+2\alpha\ii$ is the only solution of $\Re f(z)=0$ for $z=x+2\alpha\ii$ and $x\in[0,1]$.
 
 Thus the second end-point can occur only on the real or imaginary line. Using~\eqref{eq:chi_dif} once more, one obtains
 \[
  \Re f'(x)=-\frac{1}{2}\log\frac{2\alpha}{1-x+\sqrt{(1-x)^{2}+4\alpha^{2}}}>0, \quad \forall x\in(0,1).
 \]
 In addition, $\Re f(1)=\alpha>0$. Consequently, the second end-point is located in $(0,1)$ if and only if $\Re f(0)<0$ which is further equivalent to $0<\alpha<\alpha_{0}$. If $\alpha=\alpha_{0}$, then $\Re f(0)=0$ and hence the second end-point coincides with the origin. Finally, for $\alpha>\alpha_{0}$, the second end-point has to be located in the remaining side of~$\Omega^{(\alpha)}$ which is~$\ii(0,2\alpha)$.
\end{proof}

\begin{rem}
 Let $\xi(\alpha)$ stands for the second end-point of the curve from Lemma~\ref{lem:stokes_line_prop} located on the real or the imaginary line. Without going 
 into details, let us point out that 
 \[
  \lim_{\alpha\to0+}\xi(\alpha)=1 \quad\mbox{ and }\quad \lim_{\alpha\to+\infty}\left(\xi(\alpha)-2\alpha\ii\right)=0.
 \]
 In fact, if $0<\alpha\leq\alpha_{0}$, one can even show that
 \[
  \xi(\alpha)=1-\frac{\alpha}{\alpha_{0}}
 \]
 by verifying that $\chi_{\alpha}(-1+\alpha/\alpha_{0})=0$.
\end{rem}

\subsection{The asymptotic zero distribution}

Now we are in the position to prove our main result.

\begin{proof}[Proof of Theorem~\ref{thm:main}]
 The general identity~\eqref{eq:cauchy_root_meas} together with Corollary~\ref{cor:limit_cauchy} yields a formula for the limiting Cauchy transform:
 \begin{equation}
  \lim_{n\to\infty}C_{\mu_{n}^{(\alpha)}}(z)=C(z):=\begin{cases}-\frac{\ii\pi}{2}+h_{\alpha}(z)-h_{\alpha}(-z),&  \mbox{ for } z\in\Omega_{+}^{(\alpha)}, \\
  -\log(2\alpha)+h_{\alpha}(z)+h_{\alpha}(-z),&\mbox{ for } z\in\Omega_{-}^{(\alpha)}, \\  
   \end{cases}
 \label{eq:lim_cauchy}
 \end{equation}
 where $\mu_{n}^{(\alpha)}$ stands for the root-counting measure of~$Q_{n}^{(\alpha)}$. Taking also Lemma~\eqref{lem:local} into account, one concludes that the function~$C$ extends as an analytic function to the entire first quadrant $\{z\in\C \mid \Re z>0, \Im z>0\}$ except the cut given by the common boundary of~$\Omega_{+}^{(\alpha)}$ and~$\Omega_{-}^{(\alpha)}$ or equivalently by solutions of the equation~\eqref{eq:stokes_line} in $\Omega^{(\alpha)}$. This equation coincides with~\eqref{eq:impl_supp_mu_1}.
 
 To extend~$C$ beyond the first quadrant, one uses the symmetry relations~\eqref{eq:symm_Q_n} which imply the limiting function~$C$ extends according to the equations
 \begin{equation}
  \overline{C(z)}=C(\overline{z}) \quad\mbox{ and }\quad C(-z)=-C(z).
  \label{eq:symm_cauchy}
 \end{equation}
  In total, we see that the sequence $C_{\mu_{n}}^{(\alpha)}$ converges to the function $C$, as $n\to\infty$, which is analytic everywhere except the cut given by~\eqref{eq:stokes_line} in the first quadrant, its three symmetric copies in the remaining three quadrants, and possibly in certain subsets of the real and imaginary line. As a result, $\mu_{n}^{(\alpha)}$ converges weakly to a measure~$\mu^{(\alpha)}$ whose Cauchy transform coincides with $C$, i.e., $C=C_{\mu^{(\alpha)}}$ in the domain of analyticity.
 
 The measure~$\mu^{(\alpha)}$ is determined by the discontinuities of~$C$. Bearing in mind the symmetry of $\mu^{(\alpha)}$ with respect to the axes and Lemma~\ref{lem:local}, we may restrict the analysis again to the rectangle $[0,1]+\ii[0,2\alpha]$. Thus, one can write the measure $\mu^{(\alpha)}$ restricted to~$[0,1]+\ii[0,2\alpha]$ as the sum $\mu_{1}^{(\alpha)}+\mu^{(\alpha)}_{2}$, 
 where the measure $\mu^{(\alpha)}_{1}$ is induced by the discontinuity of $C$ on the cut given by~\eqref{eq:stokes_line} and $\mu^{(\alpha)}_{2}$ by the possible discontinuity of $C$ in $[0,1]\cup\ii[0,2\alpha]$.
 
 First, we discuss~$\mu^{(\alpha)}_{1}$. Suppose that the curve determined by~\eqref{eq:stokes_line} is oriented such that it starts at the point $\xi(\alpha)\in[0,1)\cup\ii[0,2\alpha)$ and ends at the point $1+2\alpha\ii$, see Lemma~\ref{lem:stokes_line_prop}. Then the set~$\Omega^{(\alpha)}_{\pm}$ lies on the $\pm$-side of the curve and, by~\eqref{eq:lim_cauchy} and~\eqref{eq:def_h}, we have
 \[
  C(z+)-C(z-)=-\frac{\pi\ii}{2}-\log\left(\frac{1}{2\alpha}\left(1-z+\sqrt{(1-z)^{2}+4\alpha^{2}}\right)\right),
 \]
 for $z$ lying on the curve. The function on the right-hand side is bounded in the entire closure of $\Omega^{(\alpha)}$ and hence integrable. Thus, by~\eqref{eq:plemelj-sokhotski}, the measure $\mu^{(\alpha)}_{1}$ is absolutely continuous on the curve and 
 one arrives at the formula~\eqref{eq:mu_1_dens} for its density.

 Second, we discuss $\mu_{2}^{(\alpha)}$ assuming that $\alpha<\alpha_{0}$. By Lemma~\ref{lem:stokes_line_prop}, $\xi(\alpha)\in(0,1)$. Then, using formulas~\eqref{eq:lim_cauchy} and~\eqref{eq:symm_cauchy}, one readily checks that $C$ extends continuously to $[0,1]\cup\ii[0,2\alpha]$ except the interval~$[0,\xi(\alpha)]$, where it has the jump
 \[
  C(x+)-C(x-)=2\ii\lim_{\epsilon\to0+}\Im C(x+\ii\epsilon)=-\ii\pi,
 \]
 for $0<x<\xi(\alpha)$. Thus, the application of~\eqref{eq:plemelj-sokhotski} yields the claim~(i) of Theorem~\ref{thm:main}.
 
 Next, suppose that $\alpha>\alpha_{0}$. Then $\xi(\alpha)\in\ii(0,2\alpha)$ by Lemma~\ref{lem:stokes_line_prop}, and, using formulas~\eqref{eq:lim_cauchy} and~\eqref{eq:symm_cauchy}, one verifies that the only discontinuity of~$C$ in $[0,1]\cup\ii[0,2\alpha]$ occurs in the line segment $\ii[0,\Im \xi(\alpha)]$, where 
 \[
  C(\ii y+)-C(\ii y-)=-2\lim_{\epsilon\to0+}\Re C(\epsilon+\ii y)=-2\log\left(\frac{1}{2\alpha}\left|1+\ii y+\sqrt{(1+\ii y)^{2}+4\alpha^{2}}\right|\right),
 \]
  for $0<y<\Im\xi(\alpha)$. Taking into account the obvious integrability of the above function for $y\in[0,\Im\xi(\alpha)]$, one arrives at the claim~(ii) of Theorem~\ref{thm:main} by applying~\eqref{eq:plemelj-sokhotski}.
  
  In the particular case when $\alpha=\alpha_{0}$, $\xi(\alpha)=0$ by Lemma~\ref{lem:stokes_line_prop} and the function~$C$ extends continuously to the axes except the origin. As a result, $\mu_{2}^{(\alpha)}=0$ which proves the claim~(iii) and completes the proof of Theorem~\ref{thm:main}.
\end{proof}

\section{Numerical illustrations}\label{sec:numer}

In this section, asymptotic properties of the distribution of zeros of polynomials $Q_{n}^{(\alpha)}$, as $n\to\infty$, are numerically illustrated in several plots. 
First, Figure~\ref{fig:rootplot} 
%Figures~\ref{fig:rootplot-a},~\ref{fig:rootplot-b},~\ref{fig:rootplot-c} 
shows limiting curves of the roots, i.e., supports of measures $\mu^{(\alpha)}$ from Theorem~\ref{thm:main} in three different regimes.

\begin{figure}[htb!]
    \centering
    \begin{subfigure}[b]{0.32\textwidth}
        \includegraphics[width=\textwidth]{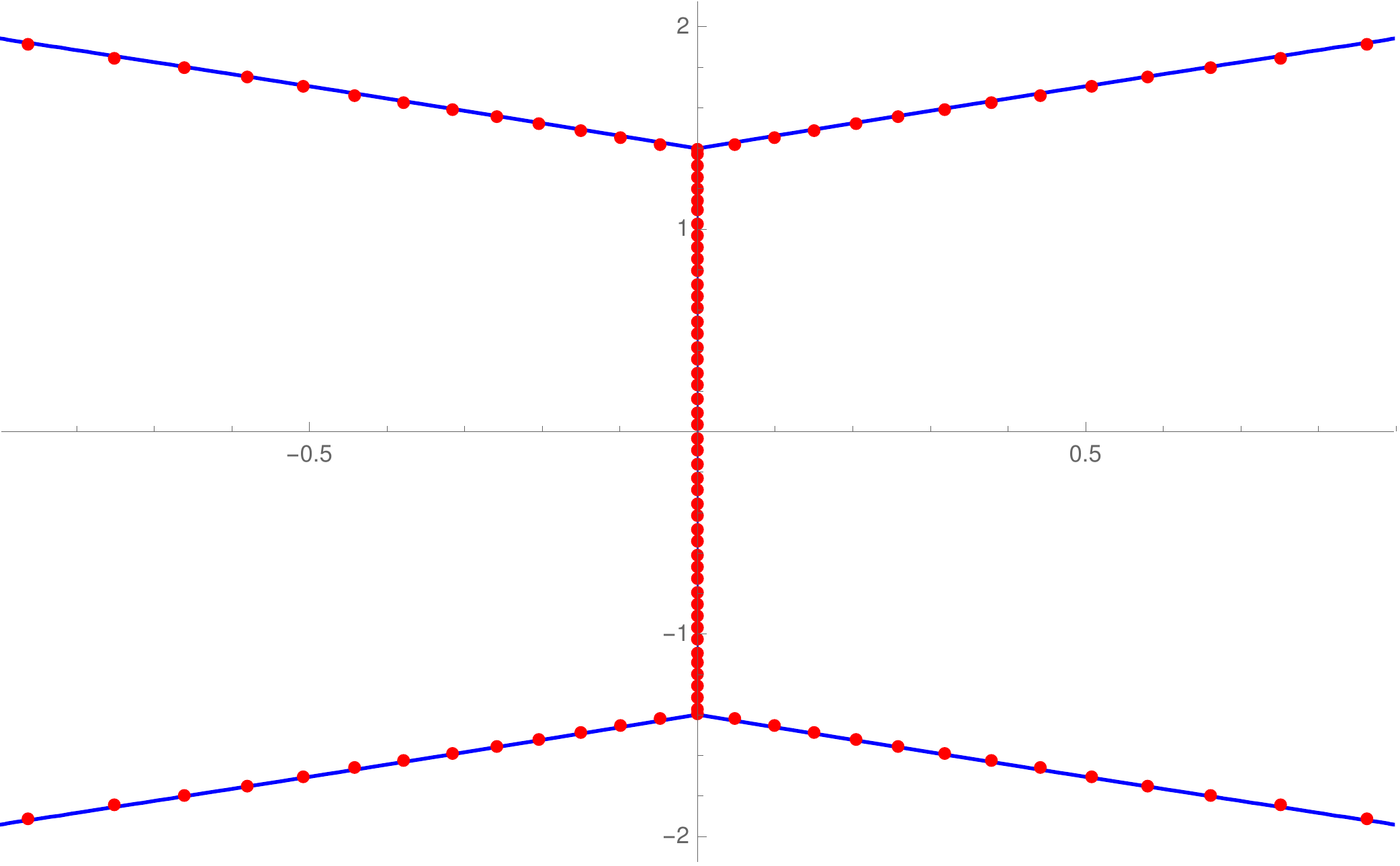}
        \caption{\centering $\alpha=1$}
    \end{subfigure}
    \begin{subfigure}[b]{0.32\textwidth}
        \includegraphics[width=\textwidth]{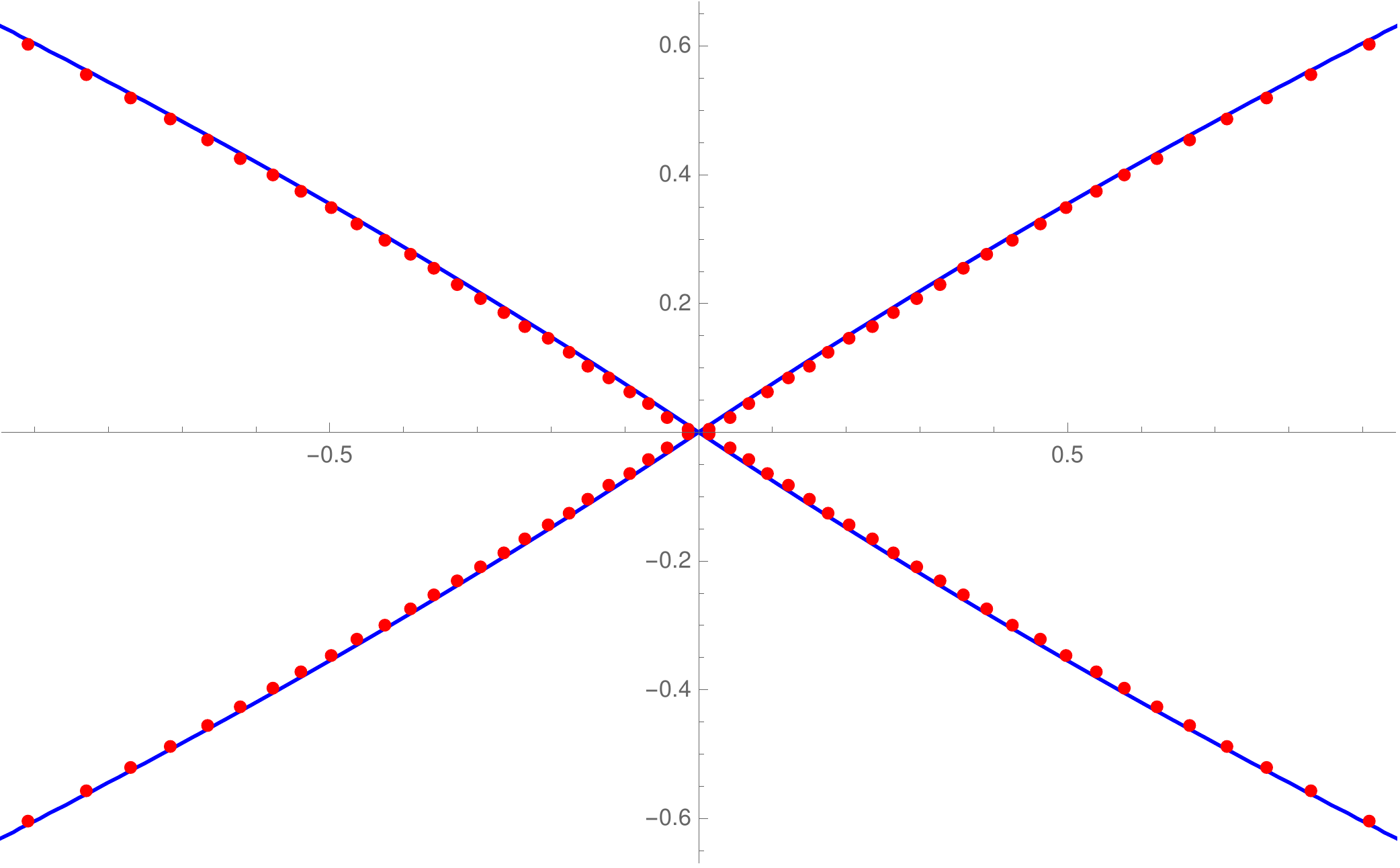}
        \caption{\centering $\alpha=\alpha_{0}\approx0.33137$}
    \end{subfigure}
    \begin{subfigure}[b]{0.32\textwidth}
        \includegraphics[width=\textwidth]{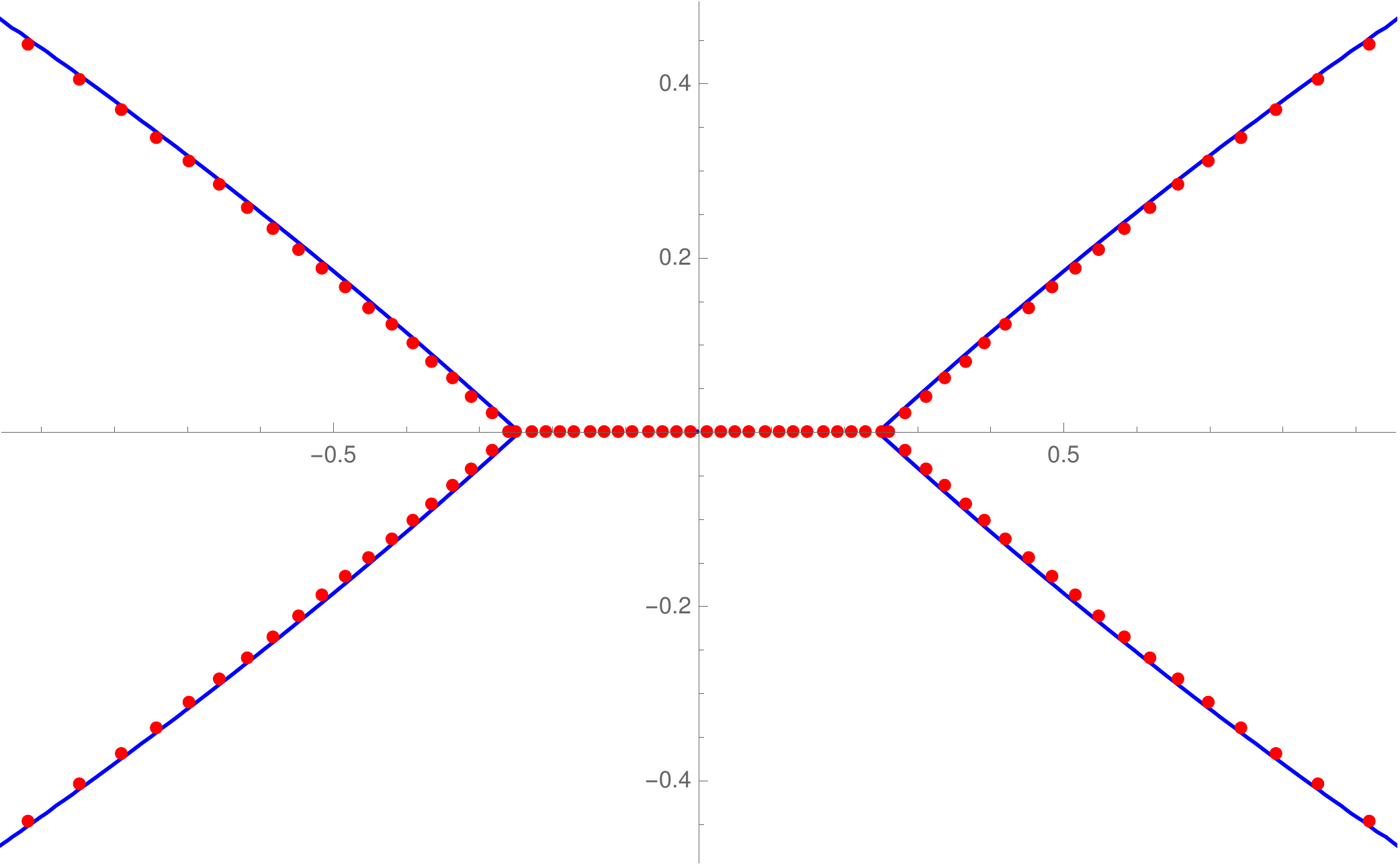}
        \caption{\centering $\alpha=0.25$}
    \end{subfigure}
    \caption{Red dots are roots of $Q_{n}^{(\alpha)}$ for $n=100$. Blue lines show the support of the corresponding asymptotic zero distribution.}
    \label{fig:rootplot}
\end{figure}

% \begin{figure}[htb!]
% 	\centering
% 	\includegraphics[width=0.9\textwidth]{pic/rootplot-a.pdf}
% 	\caption{Red dots are roots of $Q_{n}^{(\alpha)}$ for $\alpha=1$ and $n=100$. Blue lines show the support of the corresponding asymptotic zero distribution.}
% 	\label{fig:rootplot-a}
% \end{figure}
% 
% \begin{figure}[htb!]
% 	\centering
% 	\includegraphics[width=0.9\textwidth]{pic/rootplot-b.pdf}
% 	\caption{Red dots are roots of $Q_{n}^{(\alpha)}$ for $\alpha=\alpha_{0}\approx0.33137$ and $n=100$. Blue lines show the support of the corresponding asymptotic zero distribution.}
% 	\label{fig:rootplot-b}
% \end{figure}
% 
% \begin{figure}[htb!]
% 	\centering
% 	\includegraphics[width=0.9\textwidth]{pic/rootplot-c.pdf}
% 	\caption{Red dots are roots of $Q_{n}^{(\alpha)}$ for $\alpha=0.25$ and $n=100$. Blue lines show the support of the corresponding asymptotic zero distribution.}
% 	\label{fig:rootplot-c}
% \end{figure}

Second, for two choices of the parameter~$\alpha$, Figure~\ref{fig:density} shows the densities $\mu_{1}^{(\alpha)}$ and~$\mu_{2}^{(\alpha)}$ defined in Theorem~\ref{thm:main} and compare them with the corresponding histograms for distributions of roots of $Q_{n}^{(\alpha)}$ for $n=500$. Note that it is by no means obvious from~\eqref{eq:mu_1_dens} that $\mu_{1}^{(\alpha)}$ is a positive measure on the arc in $\Omega^{(\alpha)}$ given by~\eqref{eq:impl_supp_mu_1}. We parametrize the curve defined by~\eqref{eq:impl_supp_mu_1}, say~$\gamma_{1}$, by the real variable. It means that $\gamma_{1}(x)=x+\ii y(x)$, where $y=y(x)$ is implicitly defined by~\eqref{eq:impl_supp_mu_1} for $x\in(0,1)$, if $\alpha\geq\alpha_{0}$, or $x\in(\xi(\alpha),1)$, if $\alpha<\alpha_{0}$. Then straightforward manipulations of~\eqref{eq:impl_supp_mu_1} and~\eqref{eq:mu_1_dens} allow to express the density of~$\mu_{1}^{(\alpha)}$ in the form
\begin{equation}
 \frac{\dd\mu_{1}^{(\alpha)}}{\dd x}(x)=\frac{1}{\pi}\frac{\left(\log|\mathcal{Y}(x)|\right)^{2}+\left(\frac{\pi}{2}+\arg|\mathcal{Y}(x)|\right)^{2}}{\pi+2\arg\mathcal{Y}(x)},
\label{eq:mu_1_dens_positive_form}
\end{equation}
where
\[
 \mathcal{Y}(x)=\frac{1}{2\alpha}\left(1-x-\ii y(x)+\sqrt{(1-x-\ii y(x))^{2}+4\alpha^{2}}\right),
\]
from which the positivity of $\mu_{1}^{(\alpha)}$ can be readily seen. The formula~\eqref{eq:mu_1_dens_positive_form} is used in plots of Figure~\ref{fig:density} in~(B) and~(D). Concerning the density of the measure $\mu_{2}^{(\alpha)}$ supported on the real or imaginary line segment, we use the exact forms given in Theorem~\ref{thm:main} in plots of Figure~\ref{fig:density} in~(A) and~(C).

\begin{figure}[htb!]
    \centering
    \begin{subfigure}[b]{0.45\textwidth}
        \includegraphics[width=\textwidth]{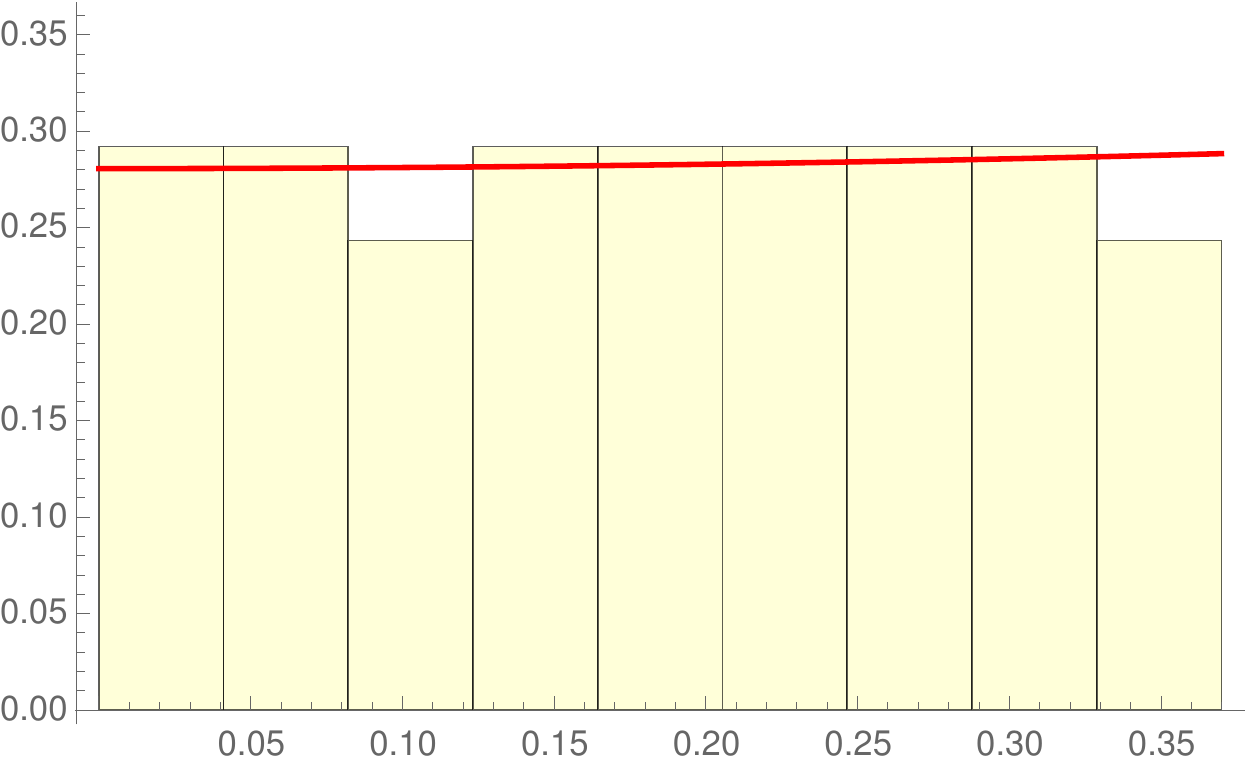}
        \caption{Density~$\dd\mu_{2}^{(0.5)}/\dd y$ and histogram of roots of $Q_{500}^{(0.5)}$ localized in the imaginary line segment~$\ii[0,\Im\xi(0.5)]\approx\ii[0,0.369]$.}
    \end{subfigure}
    \begin{subfigure}[b]{0.04\textwidth}
    $\,$
    \end{subfigure}
    \begin{subfigure}[b]{0.45\textwidth}
        \includegraphics[width=\textwidth]{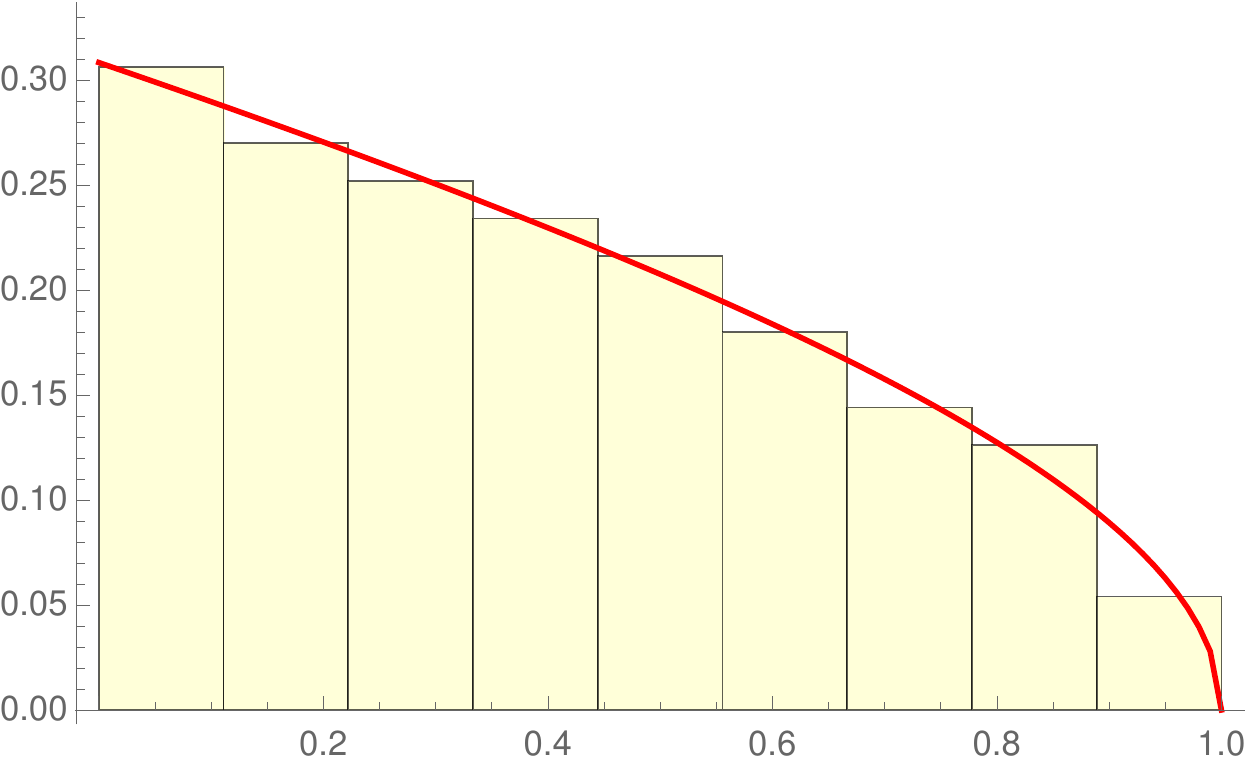}
        \caption{Density~$\dd\mu_{1}^{(0.5)}/\dd x$ and histogram of roots of $Q_{500}^{(0.5)}$ localized in~$(0,1)+\ii(0,1)$.\\
        }
    \end{subfigure}
    
    \begin{subfigure}[b]{0.45\textwidth}
        \includegraphics[width=\textwidth]{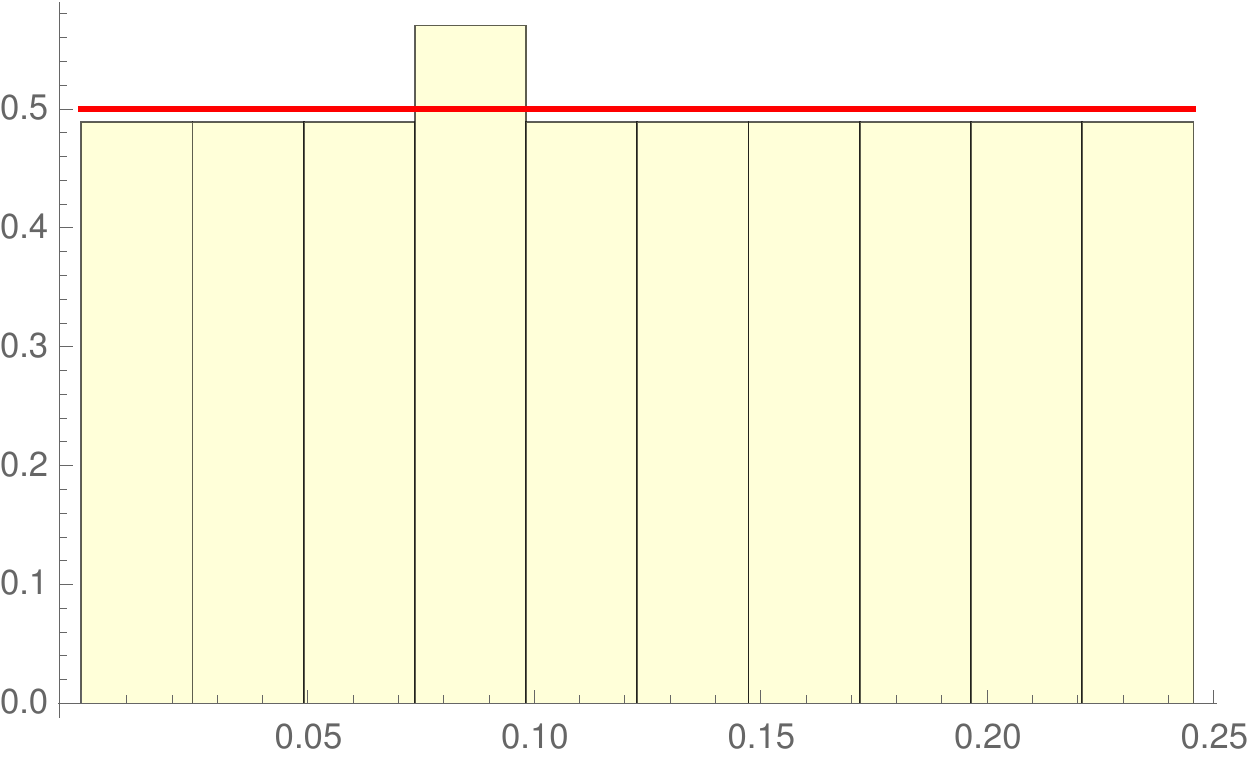}
        \caption{ The uniform density~$\dd\mu_{2}^{(0.25)}/\dd x=0.5$ and histogram of roots of $Q_{500}^{(0.25)}$ localized in the interval $[0,\xi(0.25)]\approx[0,0.246]$.
        }
    \end{subfigure}
    \begin{subfigure}[b]{0.04\textwidth}
    $\,$
    \end{subfigure}
    \begin{subfigure}[b]{0.45\textwidth}
        \includegraphics[width=\textwidth]{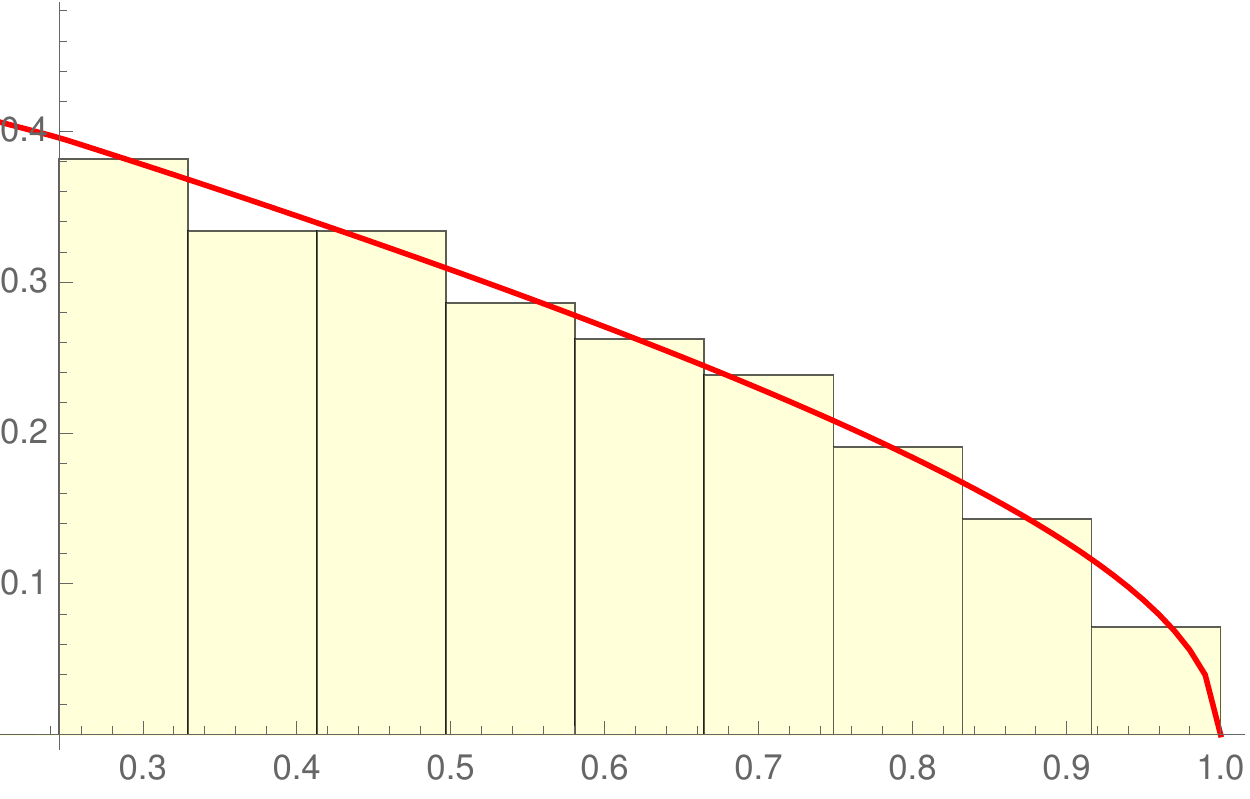}
        \caption{Density~$\dd\mu_{1}^{(0.25)}/\dd x$ and histogram of roots of $Q_{500}^{(0.5)}$ localized in~$(0,1)+\ii(0,0.5)$.\\}
    \end{subfigure}
    \caption{For $\alpha=0.25$ (top) and $\alpha=0.5$ (bottom), densities of $\mu_{1}^{(\alpha)}$ (right) and $\mu_{2}^{(\alpha)}$ (left) are plotted in red and compared with the corresponding distribution of roots of $Q_{n}^{(\alpha)}$ for $n=500$.}
    \label{fig:density}
\end{figure}

Lastly, we illustrate an evolution of limiting curves on which the roots of $Q_{n}^{(\alpha)}$ cluster, as $n\to\infty$, i.e., $\supp\mu_{1}^{(\alpha)}\cup\supp\mu_{2}^{(\alpha)}$, when $\alpha$ is increasing. For any $\alpha>0$, the limiting curve is always a union of an arc connecting $1+2\ii\alpha$ with the intersection point~$\xi(\alpha)$ and the line segment between $\xi(\alpha)$ and the origin. In order to compare the limiting curves for different values of~$\alpha$, we scale the imaginary part of the variable by $1/(2\alpha)$ to keep the curves in the fixed domain $[0,1]+\ii[0,1]$ independent of~$\alpha$. This is plotted in~Figure~\ref{fig:limcurves}.

\begin{figure}[htb!]
	\centering
	\includegraphics[width=0.9\textwidth]{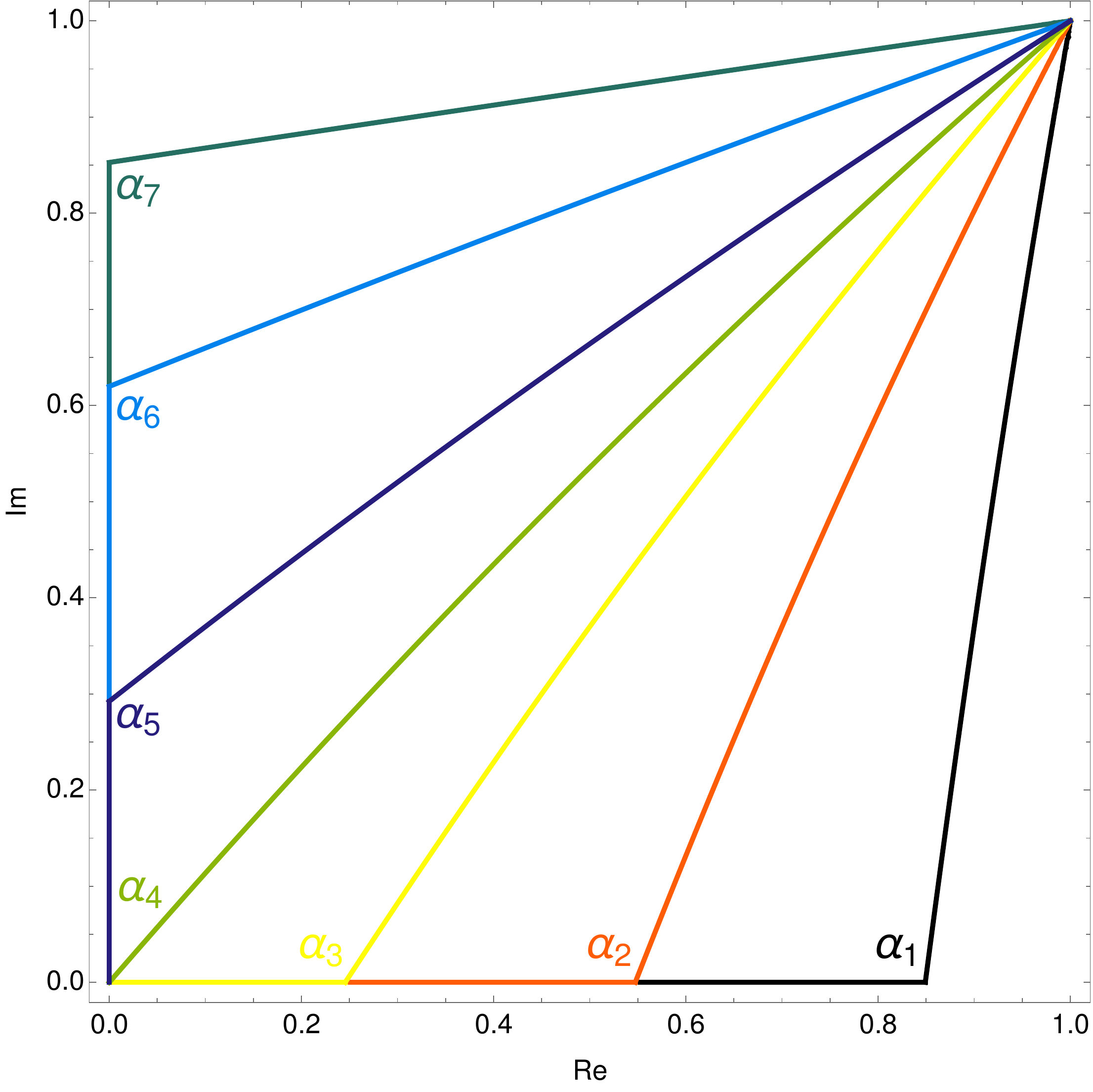}
	\caption{The limiting curves, i.e., $\supp\mu_{1}^{(\alpha)}\cup\supp\mu_{2}^{(\alpha)}$, with the imaginary part of the variable scaled by $1/(2\alpha)$ plotted for 7 values of~$\alpha$: $\alpha_{1}=0.05$, $\alpha_{2}=0.15$, $\alpha_{3}=0.25$, $\alpha_{4}\equiv\alpha_{0}\approx0.33$, $\alpha_{5}=0.45$, $\alpha_{6}=0.8$, and $\alpha_{7}=2$,}
	\label{fig:limcurves}
\end{figure}

\section{Asymptotic eigenvalue distribution of complex sampling Jacobi matrices}\label{sec:appl}

Theorem~\ref{thm:main} represents a solvable instance of a more general problem. Suppose $a,b:[0,1]\to\C$ are given continuous functions and define a sequence of Jacobi matrices whose diagonals are determined by sampling the values of functions $a,b$ on the regular grid as follows:
\begin{equation}
 J_{n}(a,b):=\begin{pmatrix} b\left(\frac{1}{n}\right) & a\left(\frac{1}{n}\right)\\[4pt]
 a\left(\frac{1}{n}\right) & b\left(\frac{2}{n}\right) & a\left(\frac{2}{n}\right)\\[4pt]
 & a\left(\frac{2}{n}\right) & b\left(\frac{3}{n}\right) & a\left(\frac{3}{n}\right)\\[4pt]
 & & \ddots & \ddots & \ddots\\[4pt]
 & & & a\left(\frac{n-2}{n}\right) & b\left(\frac{n-1}{n}\right) & a\left(\frac{n-1}{n}\right)\\[4pt]
 & & &  & a\left(\frac{n-1}{n}\right) & b\left(1\right) \end{pmatrix},
\label{eq:def_J_n} 
\end{equation}
for $n\in\N$. To~$J_{n}(a,b)$, one associates the corresponding eigenvalue-counting measure~$\mu_{n}(a,b)$ which is nothing but the root-counting measure of the characteristic polynomial of $J_{n}(a,b)$. The open problem is whether the weak limit of~$\mu_{n}(a,b)$ exists as $n\to\infty$. Moreover, one naturally expects that the limiting measure~$\mu(a,b)$, provided that it exists, should be describable in terms of the functions~$a$ and~$b$ in a certain way. 

This problem has been solved for real-valued functions $a$, $b$, i.e., for self-adjoint Jacobi matrices $J_{n}(a,b)$. Using different notation and slightly more general setting, the authors of~\cite{kui-van_99} deduced the asymptotic zero distribution of polynomials~$p_{n}^{(n)}$ determined by the recurrence
\begin{equation}
 p_{k}^{(n)}(z)=\left(z-b\left(\frac{k}{n}\right)\right)p_{k-1}^{(n)}(z)-a^{2}\left(\frac{k-1}{n}\right)p_{k-2}^{(n)}(z), \quad k=1,2,\dots,n,
\label{eq:recur_monic-ogp}
\end{equation}
with initial conditions $p_{-1}^{(n)}(z)=0$ and $p_{0}^{(n)}(z)=1$. It is easy to see that $p_{n}^{(n)}(z)=\det(z-J_{n}(a,b))$. Theorem~\ref{thm:main} describes~$\mu(a,b)$ in a simple but non-self-adjoint setting when
\[
 a(x)=\ii\alpha \quad \mbox{ and } \quad b(x)=2x-1
\]
since, in this particular case, one has
\[
 p_{n}^{(n)}(z)=\left(-\ii\alpha\right)^{n}Q_{n}^{(\alpha)}\left(z-\frac{1}{n}\right),
\]
for all $n\in\N$ and $z\in\C$, see the proof of Lemma~\ref{lem:local}.

The three-term recurrence~\eqref{eq:recur_monic-ogp} is relevant in context of orthogonal polynomials. However, there is no need to restrict the problem to tridiagonal matrices only. Considering matrices with more nonzero diagonals sampled by values of given functions leads to the case of generalized Toeplitz matrices (also referred to as locally Toeplitz, variable coefficient Toeplitz, Kac--Murdock--Szeg{\H o}, etc.). Under some assumptions, first of all the self-adjointness, the asymptotic eigenvalue distribution of generalized Toeplitz matrices was deduced already by Kac, Murdock, and Szeg{\H o} in~\cite{kac-mur-sze_53} and rediscovered later by Tilli~\cite{til_98}. Later on, this research led to a development of the theory of generalized locally Toeplitz sequences with applications in numerical analysis of differential equations~\cite{serra-capizzano17, serra-capizzano18}. 

Relaxing the self-adjointness assumption, it seems that the asymptotic eigenvalue distribution is known only for banded Toeplitz matrices~\cite{bot-gru_05, hir_67, sch-spi_60}, Toeplitz matrices with rational symbols~\cite{day_75a, day_75b} or in very particular cases. Few more works relevant in the non-self-adjoint setting are~\cite{golinski-serra-capizzano_07, til_99}. Recently, the authors of~\cite{bou-loy-tyl_18} conjectured, see Problem~3, that the asymptotic eigenvalue distribution of certain non-self-adjoint generalized banded Toeplitz matrices exists and its support equals a union of a finite number of pairwise disjoint open analytic arcs and a finite number of certain exceptional points - a situation familiar from the case of complex banded Toeplitz matrices~\cite{bot-gru_05}. Theorem~\ref{thm:main} is in agreement with this conjecture.

\begin{rem}
As kindly pointed out by an anonymous reviewer, an approach based on trajectories of quadratic differentials and critical measures, which was successfully applied in problems similar to the one solved in this paper, could provide an alternative view on the measure from Theorem~\ref{thm:main}. The reader may consult article~\cite{mar-fin-rak_16} as well as~\cite{mar-fin-rak_11,rakhmanov_12} for more details on the so-called GRS theory. Although the applicability of this approach is fairly general, it can be applied readily when a non-Hermitian orthogonality, with a density which is even allowed to vary with the index (see~\cite{gon-rak_87}), is available rather than when a recurrence with varying coefficients such as~\eqref{eq:recur_monic-ogp} is the starting point. No kind of non-Hermitian orthogonality was observed for the family $Q_{n}^{(\alpha)}$ though. From broader perspective, it would be very interesting if ideas of the GRS theory may also cast some light on the open spectral-theoretic problems for structured non-Hermitian matrices with varying entries.
\end{rem}

\section*{Acknowledgement}
The research of P.~B. was supported by GA{\v C}R grant No. 201/12/G028.
F.~{\v S}. acknowledges financial support by the GA{\v C}R grant No. 20-17749X.

\bibliographystyle{acm}

\end{document}